\documentclass[12pt,reqno]{amsart}
\usepackage[headings]{fullpage}
\usepackage{amssymb,amsmath,amscd,bbm,tikz}
\usepackage{graphicx}
\usepackage{texdraw}
\usepackage{pb-diagram}
\usepackage[all,cmtip]{xy}
\usepackage{url}
\usepackage[bookmarks=true,%
    colorlinks=true,%
    linkcolor=blue,%
    citecolor=blue,%
    filecolor=blue,%
    menucolor=blue,%
    urlcolor=blue,%
    breaklinks=true]{hyperref}
\usepackage{slashed}    
\usepackage{tikz-cd}
\usepackage{stmaryrd}
\usepackage{cancel}

\usepackage{soul,xcolor}

\usepackage{verbatim}

\newtheorem{theorem}{Theorem}[section]
\theoremstyle{definition}

\newtheorem{lemma}[theorem]{Lemma}

\newtheorem{remark}[theorem]{Remark}

\newtheorem{question}[theorem]{Question}

\def\BZ{\mathbbm Z}
\def\BQ{\mathbbm Q}

\def\BC{\mathbbm C}

\def\calS{\mathcal S}

\def\s{\sigma}
\def\la{\langle}
\def\ra{\rangle}

\def\SL{\mathrm{SL}}
\def\longto{\longrightarrow}

\def\a{\alpha}

\def\ve{\varepsilon}

\def\PSL{\mathrm{PSL}}

\def\be{\begin{equation}}
\def\ee{\end{equation}}

\def\z{\zeta}

\usepackage{accents}

\def\hab{\widehat{\BZ[q]}}
\def\vphi{\varphi}
\def\CS{{\mathcal S}}
\def\fsl{\mathfrak{sl}}

\def\sV{\mathsf{V}}
\def\Span{\mathrm{Span}}
\def\lk{\mathrm{lk}}
\def\SO{\mathrm{SO}}
\def\SU{\mathrm{SU}}

\newcommand\incl[2]{{\includegraphics[height=#1]{draws/#2.eps}}}

\def\cross{  \raisebox{-8pt}{\incl{.8 cm}{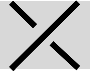}} }
\def\resoP{  \raisebox{-8pt}{\incl{.8 cm}{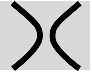}} }
\def\resoN{  \raisebox{-8pt}{\incl{.8 cm}{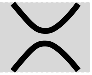}} }
\def\trivloop{  \raisebox{-8pt}{\incl{.8 cm}{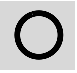}} }
\def\emptyr{\raisebox{-8pt}{\incl{.8 cm}{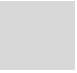}}}

\def\pM{{\partial M}}
\def\SM{{\mathcal S}(M)}

\def\WRT{{\mathsf{WRT}}}
\def\habq{{ \hab \ot _{\BZ[q^{\pm 1}]   } \Zq  }}

\def\no#1{}

\def\Zq{{ \BZ[q^{\pm 1/4}]}}
\def\Zqq{{\BZ[q^{\pm 1}]}}
\def\Qq{{ \BQ(q^{ 1/4})}}
\def\Rep{{\mathsf{Rep}}}
\def\RepQ{{\Rep_\Qq}}
\def\sV{{\mathsf{V}}}
\def\ot{{\otimes}}

\makeatletter

\renewcommand\thepart{\@Roman\c@part}%
\renewcommand\part{%
   \if@noskipsec \leavevmode \fi
   \par
   \addvspace{6.7ex}%
   \@afterindentfalse
   \secdef\@part\@spart}
\def\@part[#1]#2{%
    \ifnum \c@secnumdepth >\m@ne
      \refstepcounter{part}%
      \addcontentsline{toc}{part}{Part~\thepart.\ #1}%
    \else
      \addcontentsline{toc}{part}{#1}%
    \fi
    {\parindent \z@ \raggedright
     \interlinepenalty \@M
     \normalfont
     \ifnum \c@secnumdepth >\m@ne
       \centering\large\scshape \partname~\thepart.%
       \hspace{1ex}%
     \fi%
     \large\scshape #2%
     \markboth{}{}\par}%
    \nobreak
    \vskip 4.7ex
    \@afterheading}
  \def\@spart#1{
  \refstepcounter{part}%
  \addcontentsline{toc}{part}{#1}%
    {\parindent \z@ \raggedright
     \interlinepenalty \@M
     \normalfont
     \centering\large\scshape #1\par}%
     \nobreak
     \vskip 4.7ex
     \@afterheading}
\renewcommand*\l@part[2]{%
  \ifnum \c@tocdepth >-2\relax
    \addpenalty\@secpenalty
    \addvspace{0.75em \@plus\p@}%
    \begingroup
      \parindent \z@ \rightskip \@pnumwidth
      \parfillskip -\@pnumwidth
      {\leavevmode
       \normalsize \bfseries #1\hfil \hb@xt@\@pnumwidth{\hss #2}}\par
       \nobreak
       \if@compatibility
         \global\@nobreaktrue
         \everypar{\global\@nobreakfalse\everypar{}}%
      \fi
    \endgroup
  \fi}

\def\l@subsection{\@tocline{2}{0pt}{2pc}{6pc}{}}
\makeatother

\begin{document}
\title[From 3-dimensional skein theory to functions near $\BQ$]{
       From 3-dimensional skein theory to functions near $\BQ$}
\author{Stavros Garoufalidis}
\address{
  International Center for Mathematics, Department of Mathematics \\
  Southern University of Science and Technology \\
  Shenzhen, China \newline
  {\tt \url{http://people.mpim-bonn.mpg.de/stavros}}}
\email{stavros@mpim-bonn.mpg.de}
\author{Thang T.Q. L\^e}
\address{School of Mathematics \\
         Georgia Institute of Technology \\
         Atlanta, GA 30332-0160, USA \\
{\tt http://www.math.gatech} \newline {\tt .edu/$\sim$letu } }
\email{letu@math.gatech.edu}

\thanks{
  {\em Key words and phrases}: knots, links, 3-manifolds, Jones polynomial,
  Kauffman bracket, skein theory, Witten's Conjecture, Habiro ring, Quantum
  Modularity Conjecture, Volume Conjecture, functions ``near'' $\BQ$,
  character varieties.
}

\date{17 July 2023}

\begin{abstract}
  Motivated by the Quantum Modularity Conjecture and its arithmetic aspects
  related to the Habiro ring of a number field, we define a map from the
  Kauffman bracket skein module of an integer homology 3-sphere to the Habiro ring,
  and use Witten's conjecture (now a theorem) to show that the image is an
  effectively computable module of finite rank that can be used to phrase the
  quantum modularity conjecture.
\end{abstract}

\maketitle

{\footnotesize
\tableofcontents
}


\section{Introduction}
\label{sec.intro}

\subsection{Skein theory, the Habiro ring and the Quantum Modularity
  Conjecture}

The paper concerns a connection between three concepts of rather different
nature:
\begin{itemize}
\item[(a)]
  3-dimensional skein theory on a closed 3-manifold,
\item[(b)]
  functions near $\BQ$ and,
\item[(c)]
  the Quantum Modularity Conjecture.
\end{itemize}

The first of them, the Kauffman bracket skein module $\calS(M)$ of a 3-manifold
$M$, goes back to Przytycki~\cite{Przytycki} and Turaev~\cite{Tu:conway}
in the 1990s. It is a (non-finitely generated) $\BZ[q^{\pm 1/4}]$-module associated
to a 3-manifold whose generators are framed links, in the spirit of Conway’s
interpretation of knot invariants and the skein theory of the Jones
polynomial. This fruitful idea has been studied and extended by many authors
including Bonahon--Wong, Frohman-Kania-Bartoszynska and the second
author~\cite{Bonahon-Wong:inventiones,Frohman:inventiones}. 

The second involves a completion 
\be
\label{hab}
\hab = \varprojlim \;\BZ[q]/((q;q)_n), \qquad (q;q)_n := \prod_{j=1}^n (1-q^j)
\ee
of the polynomial ring $\BZ[q]$ discovered by Habiro in the early 2000s
(the so-called Habiro ring)~\cite{Habiro:completion} whose elements have remarkable
arithmetic properties. 
An element of the Habiro ring gives a function ``near'' $\BQ$, that is,
a collection of formal power series $f_\z(q) \in \BZ[\z] \llbracket q-\z \rrbracket$
for every complex root of unity $\z$ which satisfies integrality and congruence
properties and allows to arithmetically compute one series $f_\z(q)$ from another one.
Although these functions near $\BQ$ are not analytic in the usual topology of the
complex numbers (nor are these power series convergent), they are analytic from
the arithmetic point of view. The Habiro ring was motivated by the quantum 3-manifold
invariants and provides a natural home for them.
Indeed, two important quantum invariants defined on the set of complex roots of unity,
namely the Kashaev invariant of a knot in $S^3$ and the WRT invariant of an integer
homology 3-sphere associated to a simple Lie algebra can be lifted to elements of the
Habiro ring. The former statement was shown by Huynh and the second author~\cite{HuL}
(see also Habiro~\cite{Habiro:WRT}) and the latter by Habiro for
$\mathfrak{sl}_2$~\cite{Habiro:WRT}
and by Habiro and the second author for general simple Lie algebras~\cite{HL}.
The values at complex roots of unity and the coefficients of their power
series expansions of these functions ``near'' $\BQ$ (as were called
in~\cite{GZ:kashaev}) carry remarkable analytic information of the 
elements of the Habiro ring and the corresponding 3-manifolds that they come from.

And this brings us to the third concept, namely to the Quantum Modularity Conjecture
of Zagier and the first author~\cite{GZ:kashaev,GZ:qseries}. This is a more recent,
lesser known and conjectural topic which to a 3-manifold $M$ associates a matrix $J_M$
of functions ``near'' $\BQ$ (that is power series at each root of unity) with
interesting analytic and arithmetic properties. This conjecture for a 3-manifold
with torus boundary (such as a hyperbolic knot complement) implies Kashaev's
celebrated Volume
Conjecture~\cite{K95} to all orders in perturbation theory and with
exponentially small terms included. At the same time, the conjecture predicts that
the arithmetically defined power series at each root of unity glue together to
elements of the Habiro ring of a number field, allowing to recover a series at
one root of unity from another one as is investigated currently~\cite{GSZ:habiro}.

An extension of this conjecture for
closed 3-manifolds was studied by Wheeler in his
thesis~\cite{Wheeler:thesis}, and it is this case (even more the case
of an integer homology 3-sphere) that we will focus on.
For such manifolds, the rows and columns of the matrix $J_M$ are expected to be
parametrized by the space
\be
\label{XM}
X_M=\text{Hom}(\pi_1(M),\SL_2(\BC))/\!\!/\SL_2(\BC)
\ee
of $\SL_2(\BC)$-representations of $\pi_1(M)$ modulo conjugation when
$X_M$ is 0-dimensional and reduced. $X_M$ has a
distinguished element $\s_0$, the trivial representation which determines a
distinguished (top) row and (left) column of $J_M$. When $M$ is hyperbolic,
$X_M$ contains a lift $\s_1$ of the discrete faithful $\PSL_2(\BC)$-representation. 
The elements of the (top) row of $J_M$ are expected belong to the 
Habiro ring, and for a hyperbolic 3-manifold $M$, the element of $J_{M,\s_0,\s_1}$
is supposed to be the WRT invariant, considered as an element of the Habiro
ring ~\cite{Habiro:WRT}. 
In fact, $\BQ(q)$-span of the elements of the top row of $J_M$ are expected to
define a topological invariant of $M$, and all elements $J_{M,\s_0,\s}$ 
of the first row of $J_M$, like the distinguished element $J_{M,\s_0,\s_1}$, are
supposed to satisfy a generalization of the Chen-Yang Volume Conjecture to all
orders, and with exponentially small terms included. The elements
$J_{M,\s_0,\s}$ of the $\s_0$-row of $J_M$ were called ``descendants'' of the WRT
invariant of a closed 3-manifold~\cite{GZ:kashaev,Wheeler:thesis}.

All that may be interesting, however the indexing set $X_M$ of the rows and columns
of the matrix $J_M$ is an affine complex algebraic variety which is not always
zero-dimensional, nor reduced in general. Thus a general formulation of the
Quantum Modularity Conjecture for all closed 3-manifolds (and even the size of
the matrix $J_M$ or the indexing set of its rows and columns) is currently unknown.
There are hints, however, that the size of the matrix $J_M$ is related to the rank
of an $\SL_(\BC)$-version of Floer homology introduced by Abouzaid--Manolescu
~\cite{AM:SL2}.
We should stress that this conjecture arose from extensive experimental evidence
described in detail in~\cite[Sec.4]{GZ:kashaev}, and although we know concrete
examples with of hyperbolic ZHS with interesting character varieties $X_M$ to
experiment with, their complexity is beyond current computational limits.

\subsection{Our results}
\label{sub.results}

The goal of our paper is to prove a consequence of the Quantum Modularity Conjecture,
and to answer the topological nature of the rather mysterious ``descendants'' of
the WRT invariant. The answer is remarkably simple: the descendants are really
the Jones polynomial invariants of links in a fixed ambient integer homology
3-sphere, only that a link in an integer homology 3-sphere does not have a Jones
polynomial, but rather an element of the Habiro ring as was explained in~\cite{GL:is}.
However, there are infinitely many such links, whereas we want a $\Zq$-module
of finite rank. That is exactly where Witten's finiteness conjecture for the skein
module $\calS(M)$ of a closed 3-manifold~\cite{Witten:analytic-continuation} and its
resolution by Gunningham--Jordan--Safronov~\cite{Gunningham-Jordan} comes in.
Putting everything together, we obtain a $\BZ[q^{\pm 1/4}]$-module
$V_M \subset \hab_{\otimes \BZ[q]} \BZ[q^{\pm 1/4}] $
of finite rank, which is an effectively computable topological invariant. (By
``rank'' we mean the dimension of the $\BQ(q^{1/4})$-vector space
$V_M \otimes_{\BZ[q^{\pm 1/4}]} \BQ(q^{1/4})$). 

Fix an integer homology 3-sphere (in short, ZHS) $M$, that is a closed oriented
3-manifold with $H_1(M,\BZ)=0$. The map in the next theorem is essentially
given by the WRT invariant $(M,L)$ of a framed link $L\subset M$ in a ZHS $M$.

\begin{theorem}
\label{thm.1}
For every ZHS $M$, there is a map
\be
\label{vphi}
\vphi_M : \calS(M) \longto \hab_{\otimes \BZ[q]} \BZ[q^{\pm 1/4}]
\ee
of $\BZ[q^{\pm 1/4}]$-modules, whose image $V_M$ is a
$\BZ[q^{\pm 1/4}]$-module of finite rank. 
\end{theorem}

The above theorem explains the title of the paper and connects the largely
unproven Quantum Modularity Conjecture with the skein theory of 3-manifolds and
Witten's conjecture, much in the spirit of the
``Witten cylinder''~\cite[Fig.3]{GZ:kashaev}. It also fits well with recent results
by Detcherry--Kalfagianni--Sikora~\cite{DKS} who show, when $X_M$ is 0-dimensional
and satisfies further technical conditions, then 
 $\calS(M) \otimes_{\BZ[q^{\pm 1/4}]} \BQ(q^{1/4})$ is a
$\BQ(q^{1/4})$-vector space of rank $|X_M|$ much in the spirit of the Quantum
Modularity Conjecture. 

Aside from these connections, the theorem motivates several questions of arithmetic
and analytic origin, of interest on their own.

\begin{question}
\label{que.1}
If $M$ is hyperbolic, is the map $\vphi_M$ injective?
Is the image of $\vphi_M$ a finitely-generated
$\BZ[q^{\pm 1/4}]$-module?   
\end{question}

Regarding an effective computation of the image of $\vphi_M$, one can compare it
with two computable modules neither of which is a topological invariant of $M$
and both expected to be generically isomorphic to $V_M$ after perhaps tensoring
over $\BQ(q^{1/4})$. The first one 
\be
\label{V1}
V_{M,K} := \Span_{\Zq} \{ \vphi(K^{(m)}) \,|\, m \geq 0 \} \subset V_M 
\ee
is a submodule of $V_M$ defined for every framed knot $K$ in $M$, where $K^{(m)}$
denotes the $m$-th parallel of $K$ in $M$. 

The second one is defined when $M$ is obtained by surgery on a knot $K$ in $S^3$,
in which case (since $M$ is a ZHS), we have $M=S_{K,-1/b}$ for an integer $b$. 
This $\BZ[q^{\pm 1/4}]$-module $D_{b,K}$ (called a ``descendant''
in~\cite{GZ:kashaev,GZ:qseries,Wheeler:thesis}) can be defined in terms
of a multiparameter deformation of a formula for $\vphi_M(\emptyset)$ {given
in Equation~\eqref{IMbdesc} below}. Our next result relates the descendant
module with $V_M$.

\begin{theorem}
\label{thm.desc}
For all knots $K$ in $S^3$ and for all integers $b$, we have an inclusion 
\be
\label{V2}
V_{M,K} \subset D_{b,K}
\ee
of $\BZ[q^{\pm 1/4}]$-modules of finite rank, where $M=S^3_{K,-1/b}$.
\end{theorem}

Since the inclusions~\eqref{V1} and~\eqref{V2} involve
$\BZ[q^{\pm 1/4}]$-modules of finite rank, one expects them to be equalities,
after possibly tensoring with $\BQ(q^{1/4})$. But this brings next question.

\begin{question}
\label{que.2}
Given a finite set of elements of $\hab$, how can one prove that they are
linearly independent over $\BZ[q^{\pm 1}]$?
\end{question}

Despite the simplicity of the question this is a difficult problem and the
only method to resolve it that we know is transcendental passing through
asymptotics and analysis and back to number theory. More precisely,
it involves looking at the asymptotic expansion of these elements of the Habiro
ring at roots of unity (or asymptotic expansion of the coefficients of their
power series expansions in $q-\z$), and using the fact that the coefficients
of these asymptotic expansions
are algebraic numbers (in a fixed number field), deduce the linear independence
of the said elements of the Habiro ring. 
This was exactly proven by Wheeler in a sample hyperbolic ZHS $M_0=S^3_{4_1,-1/2}$
obtained by $-1/2$ Dehn-filling on the simplest hyperbolic $4_1$
knot~\cite[Sec.6.9]{Wheeler:thesis}. In this example, a rigorous computation
by Wheeler~\cite[Prop.1,Sec.2.2]{Wheeler:41} implies that the three modules
in~\eqref{V1} and~\eqref{V2} have the same rank, namely $8$. An independent
calculation of the
skein module by Detcherry--Kalfagianni--Sikora~\cite{DKS} shows that $\calS(M_0)$
also has rank $8$, thus the map $\vphi_{M_0}$ is an isomorphism after tensoring with
$\BQ(q^{1/4})$. On the other hand, the equality of the inclusions~\eqref{V1}
and~\eqref{V2} or the isomorphism of the map $\vphi_{M_0}$ is not known. This question
is similar to the problem of computing the Bloch group $B(F)$ of a number
field (or its higher $K$-theory version): these are finitely-generated abelian groups,
and it is easy to construct infinitely many elements in them, but it is hard to
show that these elements generate the Bloch groups in question. For a detailed
study of this beautiful subject, we refer the reader to Zagier~\cite{Zagier:dilog}
(see also~\cite{Belabas-Gangl}).

For completeness and even though we do not need it in our paper, we end this section
with a comment on the even skein module $\calS^\mathrm{ev}(M) \subset \calS(M)$ of $M$
generated by links with even color
(i.e., color from the root lattice of $\fsl_2(\BC)$), or equivalently from the
2-parallels of all
framed links in $M$. On the one hand, the classical limit of this skein module is
the variety of $\PSL_2(\BC)$-representations of $\pi_1(M)$ that lift to
$\SL_2(\BC)$-representations. This is a union of components of the
$\PSL_2(\BC)$-character variety of $\pi_1(M)$, see~\cite[Thm.4.1]{Culler:lifting}. 
On the other hand, the restriction of the map $\vphi$
\be
\vphi: \calS^\mathrm{ev}(M) \to \hab
\ee
(denoted by the same name) avoids the $q^{1/4}$-powers and takes values in the
Habiro ring itself. 


\section{Basics} 
\label{sec.basics}

\subsection{The colored Jones polynomial}
\label{sub.jones}

In this section we review some basic TQFT properties of the colored Jones
polynomial of a link in $S^3$ and some deeper integrality properties of it.
As we will see below, the latter is an element of the ring $\Zq$ of Laurent
polynomials in a formal variable $q^{1/4}$ with integer coefficients, whose
field of fractions is $\Qq$. We will use the notation

\begin{equation}
\label{notation}
v=q^{\frac{1}{2}}, \,\, [n]=\frac{v^n-v^{-n}}{v-v^{-1}}, \,\,
\{ k\} = v^k -v^{-k}, \,\, \{k\}!=\prod_{j=1}^k \{j\}, \,\, 
\left[ \begin{array}{cc}
    n \\ k
    \end{array}
  \right] = \frac{\{n\}!}{\{k\}! \{n-k\}!} \,.
\end{equation}

The quantized enveloping algebra $U_q(\fsl_2)$ of the Lie algebra $\fsl_2$ is
a Hopf algebra over $\Qq$ defined in a standard way in many books such
as~\cite{Kassel,KS}.

For a non-negative integer $n$ we denote by $\sV_n$ the
$(n+1)$-dimensional irreducible $U_q(\fsl_2)$-module of type 1 and abbreviate
$\sV=\sV_1$. 
We denote by $\Rep$ the $\Zq$-module freely spanned by $\sV_n$ for $n=0,1,2,\dots$.
We consider $\Rep$ as a $\Zq$-submodule of the Grothendieck ring of
$U_q(\fsl_2)$-modules with ground ring $\Zq$. Then $\Rep =\Zq[\sV]$ as a $\Zq$-algebra.
Let $\Rep_\Qq= \Rep \ot_\Zq \Qq$. 

The general operator invariant theory of links \cite{RT0,Turaev} associates to a
framed link $L$ in $S^3$ with $r$ ordered components colored by
$x_1, \dots , x_r\in \Rep$ an invariant $J_L(x_1,\dots,x_r) \in \Zq$,
known as the colored Jones polynomial of $L$. Let us briefly review a few properties
of this invariant which plays a key role in our paper. Fix a framed oriented link $L$
in $S^3$ and a distinguished component $K$ of $L$.

\begin{itemize}
\item
  Orientation reversing operation. The value of $J_L(x_1,\dots,x_r) \in \Zq$ does
  not change if we change the orientation of $K$ and at the same time change the
  color of $K$ to its dual (as $U_q(\fsl_2)$-module). Since every module of
  $U_q(\fsl_2)$ is self-dual, we see that $J_L(x_1,\dots,x_r)$ does not depend on
  the orientation of $L$. Hence we can consider $J_L(x_1,\dots,x_r)$ as an invariant
  of framed unoriented links.
\item
  Linearity. For $x,y\in \Rep$ and $a, b\in \Zq$ we have
  \be
  J_L( \dots, ax + b y, \dots)= a J_L( \dots, x, \dots) + b J_L( \dots, y, \dots)
 \label{eq.linear},
 \ee
 where the written colors are for the distinguished component.
\item
  Tensor product. For $x,y \in \Rep$, and $K$ has color $x\ot y$.
  \be
  J_L( \dots, x\ot y , \dots) = J_{ L^{(2)} }( \dots, x,y , \dots)
 \label{eq.tensor}
 \ee
 where $L^{(2)}$ is the result of replacing the framed knot $K$ by two parallel copies
 (using the framing) and leaving the remaining components of $L$ unchanged. 
\end{itemize}
 
When $x_i=\sV$ for all $i$, we denote $J_L(\sV, \dots, \sV)$ simply by $J_L$, which
we call the Jones polynomial. (Actually $J_L$ is equal to the original Jones
polynomial after a simple variable substitution.)

Following~\cite{Habiro:WRT}, define $P_k \in \Rep$ and $P'_k \in \RepQ$ by 
\be
\label{Pk}
P_k := \prod_{j=0}^{k-1} (\sV-v^{2j+1}-v^{-2j-1}), \qquad
P'_k := \frac{1}{\{k\}!} P_k \,.
\ee
The relation between $\sV_n$ and $P'_n$ was given in~\cite[Lem.6.1]{Habiro:WRT} 
\be
\label{VnP'}
\sV_n = \sum_{i=0}^{n} 
\left[ \begin{array}{cc} n+i+1 \\ 2i+1 \end{array} \right] \{i\}! P'_i
\ee
and the product of two elements $P'_n$ and $P'_m$ was given in 
~\cite[Eqn.(8.1)]{Habiro:WRT}
\be
\label{P'mn}
P'_n \, P'_m = \sum_{i=0}^{\min\{m,n\}}
\frac{\{m+n\}!}{\{i\}! \{m-i\}!\{n-i\}!} P'_{m+n-i} \,.
\ee
Using~\eqref{P'mn} it follows that for $i \leq k$, we have
\be
\label{Pasik}
\{i\}! P'_i P'_k = \sum_{s=0}^i a^s_{i,k}(q) P'_{k+s} =
\sum_{s=0}^i a^s_i(v,v^k) P'_{k+s}
\ee
where 
\be
\label{asik}
\begin{aligned}
a^s_{i,k}(q) &= \sum_{j=0}^i
\left[ \begin{array}{cc}
    s+j \\ j
    \end{array}
  \right] \prod_{j'=1}^{s+2j} \{k -j + j' \} \\
  &=
  \left[ \begin{array}{cc}
    s+j \\ j
    \end{array}
  \right] \prod_{j'=1}^{j} \{k -j + j' \}
  \prod_{j'=1}^{s} \{k + j' \}
  \prod_{j'=1}^{j} \{k +s + j' \}
\end{aligned}  
\ee
is divisible by $\{k+1\} \dots \{k+s\}$ and $a^s_i(v,u) \in \BZ[v^{\pm 1},u^{\pm 1}]$
are Laurent polynomials that satisfy $a^s_{i,k}(q)=a^s_i(v,v^k)$ 
for all $0 \leq s \leq i \leq k$.

This, together with~\eqref{VnP'} imply that for all $\ell \leq k$ we have
\be
\label{VellP'k}
\sV_\ell P'_k = \sum_{i=0}^\ell \gamma^i_{\ell,k}(q) P'_{k+i} =
\sum_{i=0}^\ell \gamma^i_\ell(v,v^k) P'_{k+i}
\ee
where $\gamma^i_{\ell,k}(q)$ are divisible by $\{k+1\} \dots \{k+i\}$, and
$\gamma^i_\ell(v,u) \in \BZ[v^{\pm 1},u^{\pm 1}]$ are Laurent polynomials that
satisfy  $\gamma^i_{\ell,k}(q)=\gamma^i_\ell(v,v^k)$ for all
$0 \leq i \leq \ell \leq k$.


We end this section by recalling an integrality statement for the colored Jones
polynomial of a link in $S^3$. This is a key ingredient for the integrality statements
of the WRT invariant discussed in Section~\ref{sub.WRT} and ultimately for the
definition of the map~\eqref{vphi}. 

Recall that an oriented 2-component link $(K,K')$ has well-defined linking numbers
\newline
$\lk(K,K')=\lk(K',K) \in \BZ$, and if in addition it is framed, then each component
has has a self-linking number $\lk(K,K), \,\, \lk(K',K') \in \BZ$. This extends to a
linking matrix $(\lk(L_i, L_j))_{i,j=1}^r$ (always symmetric) associated to a framed
oriented link $L$ with components $L_1, \dots, L_r$.

The integrality statement in the next theorem when $L$ is the empty link was proven
by Habiro~\cite[Thm.8.2]{Habiro:WRT}, and its extension in the presence of a link $L$
was given in~\cite[App.A]{BL:SO3}, although not explicitly stated. 

\begin{theorem}
\label{thm.integral1}
Fix two disjoint, framed, oriented links
$L'=(L'_1, \dots, L'_s)$ and $ L=(L_1, \dots, L_r)$ in $S^3$ 
colored by $(P'_{k_1}, \dots, P'_{k_s})$ and $(\sV_{\ell_1}, \dots, \sV_{\ell_r})$,
respectively, for nonnegative integers $k_i$ and $\ell_j$. Suppose that
\be
\label{Leven}
\lk(L'_i,L'_{i'})=0, \,\,\, (i,i'=1,\dots, s), \qquad
\ve_i:=\sum_{j=1}^r \lk(L'_i, L_j) \ell_j \in 2 \BZ, \,\,\, (i=1,\dots,s) \,.
\ee
Then we have
\be
J_{L'\sqcup L}(P'_{k_1}, \dots, P'_{k_s}, \sV_{\ell_1}, \dots, \sV_{\ell_r})
\in q^{f/4} \frac{(q^{k+1};q)_{k+1}}{1-q}\BZ[q^{\pm 1}]
\ee
where $k = \max \{ k_1, \dots, k_s\}$ and 
\be
f = \sum_{i,j=1}^r \lk(L_i, L_j) \ell_i \ell_j + 2 \sum_{i=1}^r (\lk(L_i, L_i) +1)\ell_i
+ \sum_ {i=1}^s \frac{k_i(k_i-1)}{2} \in \BZ \,.
\label{eq.fff}
\ee
\end{theorem}

We will call a colored framed oriented link $L' \sqcup L$ in $S^3$ \emph{even}
if it satisfies the condition~\eqref{Leven}.

\begin{proof}
When each $\ell_i$ is even (in which case each $\varepsilon_i$ is clearly even),
the theorem was stated  as~\cite[Theorem A.1]{BL:SO3} and proved there.
However the proof does not need each $\ell_i$ to be even, it requires only
$\varepsilon_i$ to be even. See~\cite[Theorem A.3]{BL:SO3} where $\varepsilon_i$
is the same as our $\varepsilon_i$. Hence the proof given in \cite[Appendix]{BL:SO3}
works also for our case.
\end{proof}

\subsection{Skein theory}
\label{sub.skein}

In this section we review some basic properties of skein theory and the well-known
relation between the Kauffman bracket and the colored Jones polynomial. 

Suppose $M$ is an oriented 3-manifold with possibly non-empty boundary $\pM$.
The skein module $\SM$, introduced by Przytycki~\cite{Przytycki} and
Turaev~\cite{Tu:conway}, is  the $\Zq$-module generated by isotopy classes of
framed unoriented links in $M$  modulo the Kauffman relations~\cite{Kauffman}

\begin{align}
\label{eq.skein0} \cross \ &= \ q^{1/4}\resoP + q^{-1/4} \resoN\\
\label{eq.loop0}  \trivloop\  &=\  (-q^{1/2} -q^{-1/2})\emptyr.
\end{align}
By convention, the empty set is considered a framed link. Note that our $q^{1/4}$
is equal to $A$ in \cite{Kauffman}. 

The skein modules have played an important role in low-dimensional topology and
quantum topology and they serve as a bridge between classical topology and quantum
topology. The skein modules have connections to the $\mathrm{SL}_2(\BC)$-character
variety~\cite{Bullock,PS1}, the quantum group of $\mathrm{SL}_2(\BC)$, see
eg~\cite{CL,LS:SLn}, the Witten-Reshetikhin-Turaev topological quantum field
theory~\cite{BHMV}, the quantum Teichm\"uller theory, see
e.g.~\cite{Kashaev,CF,BW,Le:QT}, and the theory of quantum cluster algebras
~\cite{Muller}. 
 
When $M=S^3$, Kauffman's theorem says that $\SM \cong \Zq$, with the isomorphism
given by $L \to \la L \ra$, where $\la L \ra$ is the Kauffman bracket of the
framed unoriented link $L$.
The relation between the Jones polynomial and the Kauffman bracket is given by
the following.

\begin{lemma}
\label{r.JonesKauff}
Let $L=(L_1, \dots, L_r)$ be a framed unoriented $k$-component link.  Then
\be
\label{JLbracket}
\la L \ra = (-1)^{\sum_{i=1}^r\lk(L_i,L_i)+1) } J_L(\sV, \dots, \sV) \,.
\ee
\end{lemma}

\begin{proof}
The statement is well-known, but for completeness we include a proof here. By 
\cite[Cor.4.13]{KM}, 
\be
\label{eq23}
\la L \ra = i^{ \sum_{i,j=1}^r \lk(L_i, L_j) } J_L(\sV, \dots, \sV)\Big |_{q^{1/4} \to -i q^{1/4}} 
\ee
By the  strong integrality of the Jones polynomial,
see~\cite[Theorem 2.2]{Le:integrality} or Theorem~\ref{thm.integral1} with
$L'= \emptyset$, we have
\be
J_L(\sV, \dots, \sV) \in q^{f/4} \Zqq, \qquad f
= \sum_{i,j=1}^r \lk(L_i, L_j) + 2 \sum_{i=1}^r(\lk(L_i,L_i)+1) \,.
\ee
The substitution $q^{1/4} \to -i q^{1/4}$ is identity on $\Zqq$ and sends $q^{f/4} $
to $(-i)^f q^f/4$. Hence from \eqref{eq23} we get \eqref{JLbracket}.
\end{proof}

\begin{lemma}
\label{r.JK}
Assume $L'=(L'_1, \dots, L'_s)$ is a colored framed oriented link  in $S^3$,
with colors $x_1, \dots, x_s\in \Rep$. There is a $\Zq$-linear map
$F: \CS(S^3\setminus L') \to \Zq$ such that if $L=(L_1,\dots, L_r)$ is a framed
unoriented $r$-component link in $S^3\setminus L'$, then 
\be
\label{eq.19}
F(L) = (-1)^{\sum_{i=1}^r\lk(L_i,L_i)+1) } J_{L' \sqcup L}( x_1, \dots, x_s, \sV, \dots, \sV ) \,.
\ee
\end{lemma}

\begin{proof}
We need to show that if we define $F(L)$ by \eqref{eq.19}, then it satisfies
the defining relations \eqref{eq.skein0} and \eqref{eq.loop0}  of the skein module.
Because $x_i\in\Rep= \Zq[\sV]$, the linearity of the colored Jones
polynomial \eqref{eq.linear} reduces the statement to the case when each $x_i$
is a power of $\sV$. Further, the tensor product formula \eqref{eq.tensor}
reduces the statement to the case each $x_i=\sV$, which will be assumed now.

From Lemma~\ref{r.JonesKauff} we have
\be
F(L) = (-1)^{\sum_{j=1}^s\lk(L'_j,L'_j)+1) } \la L' \sqcup L \ra \,.
\ee
As the Kauffman bracket satisfies the relations \eqref{eq.skein0}
and \eqref{eq.loop0} , we conclude that $F(L)$ also satisfies the satisfies
the relations \eqref{eq.skein0} and \eqref{eq.loop0}. 
\end{proof}

\subsection{The WRT invariant}
\label{sub.WRT}

In this section we review some basic properties of the WRT invariant of a framed
colored link $L$ in an oriented 3-manifold $M$ at a root of unity $\xi^{1/4}$
and its lift to an element of the Habiro ring. 

The $\SU(2)$-version $\WRT^{\SU(2)}_{(M,L)}(\xi^{1/4})\in \BC$ of the WRT invariant
was defined by Reshetikhin-Turaev~\cite{RT} as a mathematically rigorous realization
of Witten's TQFT theory. A refined version, the $\SO(3)$-invariant
$\WRT^{\SO(3)}_{(M,L)}(\xi^{1/4}) \in \BC$, was subsequently introduced by Kirby and
Melvin~\cite{KM}  and is defined when $\xi$ is primitive root of unity of odd order.
By~\cite[Prop.5.9]{BCL}, the two invariants agree when $\xi^{1/4}$ is a
primitive root of unity of odd order.

As the next theorem states, the $\SO(3)$-version of the WRT invariant
of a ZHS has slightly better integrality property than its $\SU(2)$-version. 
The following theorem is a generalization of a main result of~\cite{BL:SO3} (itself
generalizing work of Habiro~\cite{Habiro:WRT} as well as work of the
authors~\cite{GL:is,Le:strong}), and shows that the WRT invariant can be lifted to
an invariant in $\habq$ when $M$ is an integral homology 3-sphere. 

\begin{theorem}
\label{thm.surgery}
Fix a framed link $L=(L_1,\dots, L_r)$ colored by $(\sV_{\ell_1}, \dots, \sV_{\ell_r})$
in an integer homology sphere $M$ for nonnegative integers $\ell_i$.
Then, there exists a unique $I_{(M,L)}(q) \in q^{g/4} \hab$
such that for any primitive root of unity $\xi$ of odd order we have
\be
\label{eq.20a}
I_{(M,L)}(q){\big |}_{q^{1/4}=\xi^{1/4}} \, = \, \WRT^{\SO(3)}_{(M,L)}(\xi^{1/4}) \,. 
\ee
Here $\xi^{1/4}$ is any 4-th root of $\xi$,
$I_{(M,L)}(q)$ is given in terms of an even surgery presentation of $(M,L)$
by
\be
\label{eq.21a}
I_{(M,L)}(q) = \sum _{k_1, \dots, k_s=0}^\infty
J_{\mathring L' \sqcup L}(P'_{k_1}, \dots, P'_{k_s}, \sV_{\ell_1},
\dots, \sV_{\ell_r}) \prod_{i=1}^r (-d_i)^{k_i} q^{- d_i k_i(k_i+3)/4} 
\ee
where $\mathring L'$ denotes the result of changing the framing of each component of
$L'$ to $0$, $d_i= \lk(L'_i,L'_i) = \pm 1$ and
\be 
\label{eq.ggg}
g= \sum_{i,j=1}^r \lk(L_i, L_j) \ell_i \ell_j + 2 \sum_{i=1}^r (\lk(L_i, L_i) +1)\ell_i \,.
\ee 
\end{theorem}

In the proof of the theorem, by a slight abuse of notation, denote elements of $\habq$
by $F(q)$ although technically they are functions of $q^{1/4}$. This notation is also
common in the theory of modular forms as well as in proper $q$-hypergeometric series
such as Nahm sums. 

\begin{proof}
To begin with, $q$ is invertible in $\hab$ by~\cite[Prop.7.1]{Habiro:completion},
which implies that
\be
\label{eq.hab2}
\hab = \varprojlim \BZ[q^{\pm 1}]/((q;q)_n).
\ee

The uniqueness of $I_{(M,L)}(q)$ follows from the fact that an element of the Habiro
ring is uniquely determined by values at roots of 1 of odd order, as follows
from Habiro~\cite{Habiro:completion}. Let us prove the existence of $I_{(M,L)}(q)$. 

Since $M$ is a ZHS, there exists a framed link $L'=(L'_1, \dots, L'_s)$ in the 3-sphere
$S^3$ whose linking matrix is diagonal with diagonal entries
$d_i:= \lk(L'_i,L'_i) = \pm 1$ and a colored framed link $L=(L_1,\dots, L_r)
\subset S^3\setminus L'$ colored by $(\sV_{\ell_1}, \dots, \sV_{\ell_r})$ respectively,
such that surgery on $S^3$ along $L'$ transforms $(S^3, L)$ to $(M, L)$.
Sliding a component $L_j$ of $L$ over a gluing disk corresponding to the surgery along
$L'_i$ changes $\lk(L'_i, L_j)$ by $\pm 1$, but the resulting
link and the old link are isotopic in $M$. It follows that after some slidings,
we can assume that $\mathring L' \sqcup L$
is even (i.e., it satisfies~\eqref{Leven}), where
$\mathring L'$ denotes the result of changing the framing of each component of
$L'$ to $0$. 

In this case, Theorem \ref{thm.integral1} gives
\be
\label{JLL'f}
J_{\mathring L' \sqcup L}(P'_{k_1}, \dots, P'_{k_s}, \sV_{\ell_1}, \dots, \sV_{\ell_r})
\in  q^{f/4} \frac{( q^{k+1};q)_{k+1}}{1-q} \BZ[q^{\pm 1}]
\subset  q^{f/4} (q;q)_k \BZ[q^{\pm 1}] 
\ee
where $k= \max \{ k_1, \dots, k_s\}$ and $f$ is given by~\eqref{eq.fff}.

Using the explicit formula of $f$, it is easy to see  
that the summand of the right hand side of \eqref{eq.21a} is in 
$q^{g/4} (q;q)_k\,  \BZ[q^{\pm 1}]$. This, together with~\eqref{eq.hab2} implies
that $I_{(M,L)}(q) \in q^{g/4}\,  \hab.$  

Assume $\xi$ is a root of unity of odd order $N >1$. By \cite[Lemma 2.1]{Le:strong},
for $k> (N-1)/3$ we have
\be
\frac{( \xi^{k+1};\xi)_{k+1}}{1-\xi}  =0 \,.
\ee
Hence 
\be
I_{(M,L)}(q)\Big|_{q^{1/4}=\xi^{1/4}} =I_{(M,L)}^{\le (N-1)/3}(q)\Big|_{q^{1/4}=\xi^{1/4}},
\ee
where $I_{(M,L)}^{\le (N-1)/3}(q)$ is defined using the right hand side of \eqref{eq.21a}
with the sum truncated at each $k_i \le (N-1)/3$. On the other hand, we have
\be 
\WRT^{\SO(3)}_{(M,L)}(\xi^{1/4}) = I_{M,L}^{\le (N-1)/3}\Big|_{q^{1/4}=\xi^{1/4}} \,.
\ee
This was proved in~\cite[Formula (9)]{Le:strong}, where the case
$L=\emptyset$ is considered. However, the same proof also works when $L$ is a
colored framed link. Thus concludes the proof of~\eqref{eq.20a}.
\end{proof}

\begin{remark}
In her thesis~\cite{Buhler} Buhler showed that $I_{(M,L)}(q)$ recovers also the $\SU(2)$
WRT-invariant. Explicitly, for all complex roots of unity $\xi^{1/4}$, we have
\be
\label{eq.20b}
I_{(M,L)}(q)\big| _{q^{1/4} = \xi^{1/4}  } = \WRT^{\SU(2)}_{(M,L)}(\xi^{1/4}) \,. 
\ee
\end{remark}


\section{The map $\vphi$}
\label{sec.vphi}

In this section we define the map~\eqref{vphi} and prove its properties
stated in Theorem~\ref{thm.1} and its relation with the descendant knot invariants
stated in Theorem~\ref{thm.desc}.

\subsection{Definition and properties}
\label{sub.vphi}

For a framed unoriented link $L$ in $M$ define 
\be
\label{defvphi2}
\varphi(L) = (-1)^{\sum_{i=1}^r (\lk(L_i, L_i)+1)} I_{(M,L)}(q) 
\ee
where each component of $L$ is colored by the fundamental module $\sV=\sV_1$. 
In other words, $\vphi(L)$ and $I_{(M,L)}(q)$ agree up to a sign.
We will show that $\varphi(L)$ satisfies the relations \eqref{eq.skein0}
and~\eqref{eq.loop0} and hence it descends to a $\Zq$-linear map from the skein
module $\SM$ to $\habq$.

To do so, we fix framed links $L'$ and $L$ in $S^3$ of $s$
and $k$ components such that surgery along $L'$ transform $(S^3,L)$ to $(M,L)$, 
as in the proof of Theorem~\ref{thm.surgery}. Further, we color each component of $L$
with $\sV$. Using Equation~\eqref{eq.21a} we have
\begin{small}
\be
\label{eq.21b}
\varphi(L) =  \sum _{k_1,\dots,k_s=0}^\infty
\left[ (-1)^{\sum_{i=1}^r (\lk(L_i, L_i)+1)} J_{\mathring L' \sqcup L}(P'_{k_1}, \dots, P'_{k_s}, \sV,
\dots, \sV)\right] \prod_{i=1}^r (-d_i)^{k_i} q^{- d_i k_i(k_i+3)/4}. 
\ee
\end{small}
 By Lemma \ref{r.JK}, each term in the square bracket, as a function of $L$,
satisfies the defining relations \eqref{eq.skein0} and \eqref{eq.loop0}.
Hence~\eqref{eq.21b} implies that $\vphi(L)$ also satisfies
the same relations. Thus $\vphi$ descends to a $\Zq$-linear map from $\SM$ to $\habq$.

The finiteness of the rank of the image $V_M$ of $\vphi$ follows from the finiteness
of the rank of the skein module of a closed manifolds~\cite{Gunningham-Jordan}.
This completes the proof of Theorem~\ref{thm.1}.
\qed

\subsection{The relation with descendants}
\label{sub.desc}

In this section we prove Theorem~\eqref{thm.desc} in several stages. To begin
with, the Glebsh-Gordan decomposition formula implies that for every positive integer
$m$, $\sV^m$ is a $\BZ$-linear combination of $\sV_\ell$ for $\ell=1,\dots,m$ and
vice-versa.  This, together with the TQFT axioms and
Equation~\eqref{defvphi2} imply that
\be
\Span \{ \vphi(K^{(m)}) \,|\, m \geq 0 \} = \Span \{ I_{(M,K_m)}(q) \,|\, m \geq 0 \}
\ee
where $K_m$ denotes the framed knot $K$ in the ZHS $M$ colored by $\sV_m$. Thus,
we will prove that
\be
\Span \{ I_{(M,K_m)}(q) \,|\, m \geq 0 \} \subset D_{b,K} \,.
\ee

To begin with with fix an integer $b$ and a knot $K$ in $S^3$, and let $M=S^3_{K,-1/b}$
denote the corresponding ZHS. After possibly reversing the orientation
(which simply changes $q$ to $q^{-1}$ and does not affect the result), we may assume
that the Dehn filling coefficient is $-1/b$ for $b >0$. To make
the ideas as clear as possible, we start with the case $b=1$, i.e., $M=S^3_{K,-1}$
is obtained by $-1$ surgery on a 0-framed knot $K$ in $S^3$. In this case, 
the WRT invariant given in~\eqref{eq.21a} can be written in the form
\be
I_M(q) = \sum_{k \geq 0} J_K(P'_k)(q) q^{\frac{k(k+3)}{4}}
\ee
where
\be
\label{Pkint}
J_K(P'_k)(q) q^{\frac{k(k+3)}{4}} \in (q^{k+1};q)_{k+1} \BZ[q^{\pm 1}] \,,
\ee
implying that $I_M(q) \in \hab$. The above formula, although by no means
canonical, has a 1-parameter deformation (called a ``descendant''
in~\cite{GZ:kashaev,GZ:qseries,Wheeler:thesis})
\be
\label{IMm}
I^{(m)}_M(q) = \sum_{k \geq 0} J_K(P'_k)(q) q^{\frac{k(k+3)}{4}} \, q^{k m}, \qquad
(m \in \BZ)
\ee
where $I^{(m)}_M(q) \in \hab$, whose $\BZ[q^{\pm 1/4}]$-span
$$
D_{1,K}:=\Span_{\BZ[q^{\pm 1/4}]}\{1, I^{(m)}_{M}(q) \,\, | \,\, m \in \BZ \}
$$
defines the module $D_{1,K}$. We need to prove that $I_{(M,K^{(\ell)})}(q) \in
D_{1,K}$ for all $\ell \geq 0$.

When $\ell=0$, this is clear. For $\ell=1$, 
using the definition of $P_k$ as the product of $k$ terms, it follows that for
all $k$, we have
\be
\label{VP}
\sV P_k = P_{k+1} + (v^{2k+1}+v^{-2k-1}) P_k \,,
\ee
which after dividing by $\{k\}!$ implies that
\be
\label{VP'}
\sV P'_k = (v^{k+1}-v^{-k-1}) P'_{k+1} + (v^{2k+1}+v^{-2k-1}) P'_k \,.
\ee
This and the operator properties of the colored Jones polynomial reviewed in
Section~\ref{sub.jones} axioms imply that
\be
\label{JKK}
J_{K,K}(P'_k, \sV) = (v^{k+1}-v^{-k-1})
J_K(P'_{k+1}) + (v^{2k+1}+v^{-2k-1}) J_K(P'_k) \,,
\ee
which together with~\eqref{eq.21a} give 
\be
\label{IMKK}
I_{(M,K)}(q) = \sum_{k \geq 0}
((v^{k+1}-v^{-k-1}) J_K(P'_{k+1})(q) + (v^{2k+1}+v^{-2k-1})
J_K(P'_k)(q)) q^{\frac{k(k+3)}{4}} \,.
\ee
After splitting the sum into two sums, and taking care of the boundary terms,
we obtain that $I_{(M,K)}(q)$ can be written in terms of the ``descendant'' sums
as follows:
\be
\label{IMK1} 
I_{(M,K)}(q) = v^{-1}(-1+ I^{(0)}_M(q)) -v^{-3}q(-1+I^{(-1)}_M(q))+
v I^{(1)}_M(q) + v^{-1} I^{(-1)}_M(q) \,.
\ee
More generally, for a positive integer $m$, Equation~\eqref{VP'} and induction
(or Equation~\eqref{VellP'k}) implies that for all $k \geq m$, 
\be
\sV_m P'_k = \sum_{i=0}^m \gamma^i_m(v,v^k) P'_{k+j}
\ee
for $\gamma^i_m(v,u) \in \BZ[v^{\pm 1}, u^{\pm 1}]$. This and the TQFT axioms
implies that
\be
\label{IKMm}
I_{(M,K_m)}(q)
= \sum_{k \geq 0} \sum_{j=0}^m \gamma^i_m(v,v^k)
 J_K(P'_{k+j})(q) q^{\frac{k(k+3)}{2}}
\in D_{1,K} \,.
\ee
This concludes the inclusion~\eqref{V2} for $b=1$. The finiteness of the
rank of $D_{1,K}$ follows from the fact that $P_K(P'_k)$ is
$q$-holonomic~\cite{GL:qholo}, and so is the summand in~\eqref{IMm}, and hence
the sum~\cite{GL:survey}. 

We next prove Theorem~\ref{thm.desc} when $b=2$, where a new difficulty
(absent from the case of $b=1$) appears. Fix the ZHS $M=S^3_{K,-1/2}$ obtained by
$-1/2$ surgery on a knot $K$ in $S^3$. In this case, Beliakova-Le show (see
Theorem 7 and the example in the end of Section 3 of~\cite{BL}) that
\be
\label{IM2}
I_M(q) = \sum_{k \geq \ell \geq 0} J_K(P'_k) \a_{k,\ell}(q), \qquad
\a_{k,\ell}(q)= q^{\frac{k(k+3)}{4} + \ell(\ell+1)}
\frac{(q;q)_k}{(q;q)_\ell (q;q)_{k-\ell}} 
\ee
whose descendants are the 2-parameter family given by
\be
\label{IM2desc}
I_M^{(m_0,m_1)}(q) = \sum_{k \geq \ell \geq 0} J_K(P'_k) \a_{k,\ell}(q)
q^{k m_0 + \ell m_1}, \qquad (m_1,m_2 \in \BZ) 
\ee
whose $\BZ[q^{\pm 1/4}]$-span together with $1$ is denoted by $D_{2,K}$.

\noindent
The proper $q$-hypergeometric function $\a_{k,\ell}(q)$ satisfies the
linear $q$-difference equations
\be
\label{arec}
\begin{aligned}
(1-q^{k+1-\ell}) \a_{k+1,\ell} &= q^{\frac{k}{2}+1} (1-q^{k+1})\a_{k,\ell}(q) \\
(1-q^{\ell+1})\a_{k,\ell+1}(q) &=  q^{2\ell+2}(1-q^{k-\ell}) \a_{k,\ell}(q) \,.
\end{aligned}
\ee
Equations~\eqref{JKK} and~\eqref{IM2} give
\be
I_{(M,K)}(q) = \sum_{k \geq \ell \geq 0}
((v^{k+1}-v^{-k-1}) J_K(P'_{k+1}) + (v^{2k+1}+v^{-2k-1}) J_K(P'_k))
\a_{k,\ell}(q) \,.
\ee
As was done in Equation~\eqref{IMKK}, splitting the sum into two sums, it suffices
to show that each is in $D_{2,K}$. The second one is obvious, whereas for the first
one, use the first linear $q$-difference equation in~\eqref{arec} to write
$v^{k+1} \a_{k,\ell}(q)$ in terms of $\a_{k+1,\ell}(q)$. We then sum over $k$
and $\ell$ and after taking care of the the boundary terms, to deduce that the first
sum and hence $I_{(M,K)}(q)$ is in $D_{2,M}$. 

More generally, we use Equation~\eqref{VellP'k} to deduce that
\be
\label{IKMm12}
I_{(M,K_m)}(q)
= \sum_{k \geq \ell \geq 0}
\sum_{i=0}^m \gamma^i_\ell(v,v^k) J_K(P'_{k+j})(q) \a_{k,\ell}(q)
\ee
and since $\gamma^j_\ell(v,v^k)$ is divisible by $(1-q^{k+1}) \dots (1-q^{k+j})$,
it follows that we can express $\gamma^j_{\ell,k}(q) \a_{k,\ell}(q)$ as a
combination of $\a_{k+j,\ell}(q)$. Doing so, and summing over $k$, and taking
care of the boundary terms, we obtain that $I_{(M,K_m)}(q) \in D_{2,K}$ for all
nonnegative integers $m$. As in the case of $b=1$, the 2-parameter family
of descendants~\eqref{IM2desc} is $q$-holonomic, and this implies that their
span $D_{2,K}$ has finite rank, concluding the proof of Theorem~\ref{thm.desc}
when $b=2$.

The general case of $b \geq 0$ follows using the formula for the WRT invariant
of $M=S^3_{K,-1/b}$ from Theorem 7 and the example in the end of Section 3
of~\cite{BL}:
\be
\label{IMb}
\begin{aligned}
I_M(q) &= \sum_{k \geq 0} J_K(P'_k) \a^{(b)}_k(q), \\
\a^{(b)}_{k}(q) &= q^{\frac{k(k+3)}{4}}
\sum_{k \geq \ell_1 \geq \ell_2 \geq \ell_{b-1} \geq 0}
q^{\sum_{j=1}^{b-1}\ell_j(\ell_j+1)}
\frac{(q;q)_k}{(q;q)_{\ell_1} (q;q)_{\ell_1-\ell_2} \dots
  (q;q)_{\ell_{b-2}-\ell_{b-1}} (q;q)_{\ell_b}}
\end{aligned}
\ee
and the descendants are the $b$-parameter family given by
\be
\label{IMbdesc}
I_M^{(m_0,m_2,\dots,m_{b-1})}(q) = \sum_{k \geq \ell \geq 0} J_K(P'_k)
\a^{(b)}_k(q) q^{k m_0 + \sum_{j=1}^{b-1} \ell_j m_j}, \qquad (m_j \in \BZ)
\ee
whose $\BZ[q^{\pm 1/4}]$-span together with $1$ is denoted by $D_{b,K}$.
The same arguments as before conclude the proof of Theorem~\ref{thm.desc}.
\qed

\subsection{An example}
\label{sub.example}

In this section we discuss a concrete example suggested in~\cite[Example.2.5]{Ga:CS}
and analyzed in detail by Wheeler in his thesis~\cite[Sec.6.9]{Wheeler:thesis}.
Consider the ZHS $M_0=S^3_{4_1,-1/2}$ obtained by $-1/2$ surgery on the $4_1$ knot.
This is a hyperbolic manifold whose trace field has degree $7$, and whose
fundamental group has $8$ conjugacy classes of representations into $\SL_2(\BC)$.

The 2-parameter family of descendants WRT invariants is given by 
\be
\label{IM0}
(1-q) I_{M_0}^{(m_0,m_1)}(q) = \sum_{k \geq \ell \geq 0} (-1)^k
q^{-\frac{1}{2}k(k+1)+\ell(\ell+1)} \frac{(q;q)_{2k+1}}{(q;q)_\ell (q;q)_{k-\ell}}
q^{k m_0 + \ell m_1}, \qquad (m_1,m_2 \in \BZ) 
\ee
with $(m_0,m_1)=(0,0)$ corresponding to the WRT invariant $I_{M_0}(q)$. These
descendants form a $q$-holonomic system of rank $8$, which implies that their
$\BZ[q^{\pm 1/4}]$-span ($D_{2,4_1}$, in the notation of Theorem~\ref{thm.desc})
has rank at most $8$. Using asymptotics and their arithmetic properties, Wheeler
proves in~\cite{Wheeler:41} $D_{2,4_1}$ has rank is exactly $8$. Moreover, an
explicit calculation implies that the module $V_{M_0,4_1}$ in
Theorem~\ref{thm.desc} has rank $8$. This and Equation~\eqref{V1} implies
that $V_{M_0}$ has rank at least $8$. On the other hand, the skein module $\calS(M_0)$
has rank $8$, which implies that $\vphi$ is (after tensoring with $\BQ(q^{1/4})$
an isomorphism. 

\subsection*{Acknowledgements} 

The authors  wish to thank  Adam Sikora, Peter Scholze, Campbell Wheeler
and Don Zagier for enlightening conversations. The work of T.~L. is
partially supported by NSF grants DMS-1811114 and DMS-2203255.


\bibliographystyle{plain}
\bibliography{biblio}
\end{document}